\documentclass[11pt]{article}
\usepackage{float}
\usepackage{amscd}
\usepackage{amsmath,amstext,amsthm,amsbsy,amssymb}
\usepackage{multirow}
\usepackage{epsfig}
\newcommand{\bi}{{\bf i}}
\newcommand{\bj}{{\bf j}}
\newcommand{\bk}{{\bf k}}
\newcommand{\bc}{{\mathbb C}}
\newcommand{\br}{{\mathbb R}}
\newcommand{\bh}{{\mathbb H}}

\newcommand{\F}{{\mathbb F}}

\newtheorem{thm}{Theorem}[section]
\newtheorem{lem}{Lemma}[section]

\newtheorem{exam}{Example}[section]
\newtheorem{defi}{Definition}[section]

\begin{document}
\title{On left spectrum of a split quaternionic matrix}
\author{ Wensheng Cao \\
School of Mathematics and Computational Science,\\
Wuyi University, Jiangmen, Guangdong 529020, P.R. China\\
e-mail: {\tt wenscao@aliyun.com}
}
\date{}
\maketitle

\bigskip
{\bf Abstract} \,\,  In noncommutative and nondivision algebra, left spectrum of  matrices   are less known and is not easy to handle.  Split quaternion algebra is a  noncommutative and nondivision algebra.  In this paper, by the formulas of solving the equations $ax=b$ and $ax^2+bx+c=0$ over split quaternions, we make an attempt to understand the left spectrum of a split quaternion matrix of order 2.

\vspace{3mm}\baselineskip 12pt

\vspace{3mm}\baselineskip 12pt



\section{Introduction}
Let $\br$ and $\bc$ be the field of  real and  complex numbers, respectively. The quaternion algebra $\bh$  and split quaternion algebra $\bh_s$  are noncommutative extensions of the complex numbers.
Let $\F=\bh$ or $\bh_s$. Then $\F$ can be represented as
$$\F=\{x=x_0+x_1\bi+x_2\bj+x_3\bk,x_i\in \br,i=0,1,2,3\},$$
where $1,\bi,\bj,\bk$ are basis of $\F$ satisfying the following multiplication rules:

\begin{table}[h]  \centering
	\caption{The multiplication tables for $\bh$ and $\bh_s$}
	\vspace{3mm}
	\begin{tabular}{c|cccc}
		$\bh$	&1& $\bi$ &  $\bj$  &  $\bk$\\
		\hline
		1 & 1& $\bi$ &  $\bj$  &  $\bk$\\
		$\bi$ &$\bi$ &-1 & $\bk$ & -$\bj$ \\
		$\bj$&$\bj$&-$\bk$ &-1& $\bi$  \\
		$\bk$	&$\bk$&$\bj$ & $-\bi$ & -1
	\end{tabular}\quad\quad\quad
	\begin{tabular}{c|cccc}
		$\bh_s$	&1& $\bi$ &  $\bj$  &  $\bk$\\
		\hline
		1 & 1& $\bi$ &  $\bj$  &  $\bk$\\
		$\bi$ &$\bi$ &-1 & $\bk$ & -$\bj$ \\
		$\bj$&$\bj$&-$\bk$ &1& -$\bi$  \\
		$\bk$	&$\bk$&$\bj$ & $\bi$ & 1
	\end{tabular}
\end{table}

Let $\bar{x}=x_0-x_1\bi-x_2\bj-x_3\bk$, \begin{equation}\Re(x)=(x+\bar{x})/2=x_0,\,\, \Im(x)=(x-\bar{x})/2=x_1\bi+x_2\bj+x_3\bk\end{equation}
 be  respectively the conjugate,  real part  and imaginary part of $x\in\F$. Obviously we have 
  $$\F=\bc\oplus\bc\bj,\,\,\bj z=\bar{z}\bj \mbox{ for } z\in \bc.$$
That is,  each element  $x\in \F$ can be expressed  as
\begin{equation}\label{crex}x=(x_0+x_1\bi)+(x_2+x_3\bi)\bj =z_1+z_2\bj=z_1+\bj\overline{z_2},z_1,z_2\in \bc.\end{equation}
For $x\in \F$, we define
\begin{equation}\label{Ix}I_x=\bar{x}x=x\bar{x}.\end{equation}
It can be easily verified that $$ \overline{xy}=\bar{y}\bar{x},\  I_{yx}=I_yI_x, \forall x,y\in\F.$$
The set $\F^{m\times n}$ denotes all $m\times n$ type matrices over $\F$. 
The left  and right scalar multiplication is defined as
$$qA=(qa_{ij}),\, Aq=(a_{ij}q), \, \forall q\in \F,\, A=(a_{ij})\in \F^{m\times n}.$$

In noncommutative algebra $\F$, there are two types of eigenvalues.

\begin{defi}\label{def1}(c.f.\cite{ala12,ozd2013,hs2001}) Let $A\in\F^{n\times n}$ and $\lambda\in \F$.   If  $\lambda$ holds the equation $Ax =
	\lambda	x$ (resp. $Ax =
	x \lambda$) for $0\neq x\in \F^{n}$, then $\lambda$  is called a left (resp. right) eigenvalue of $A$. The set of distinct left (resp. right) eigenvalues is called the left (resp. right)	spectrum of A, denoted $\sigma_l(A)$ (resp. $\sigma_r(A)$) .
\end{defi}

In quaternion algebra, the study of left eigenvalues stemmed from the Lee  and Cohn's \cite{cohn77,lee49}  question whether left eigenvalue always exists. By using of topological method,  Wood \cite{wood85}  confirmed  that the left eigenvalue always exists.  Huang and So \cite{hs2001} reduced  the computation of $\lambda\in \sigma_l(A), A\in \bh^{2\times 2}$ to solving
a quaternionic quadratic equation \cite{hs02}.   So \cite{so05} also showed the existence of  $\lambda\in \sigma_l(A), A\in \bh^{3\times 3}$ is equivalent to the   the existence of a solution of a generalized quaternionic polynomial of degree 3, which is guaranteed by the fundamental theorem of algebra for quaternions \cite{eilen44}.
It is still an open problem whether this algebraic approach works for general $n\times n$ matrices for $n\ge 4$.

A division algebra is an algebra, where the zero element is the only noninvertible element. The quaternion algebra is a division algebra, while split quaternion algebra is a nondivision algebra.
The set of zero divisors is denoted by  \begin{equation}
	Z(\bh_s)=\{x\in \bh_s:I_x=x_0^2+x_1^2-x_2^2-x_3^2=|z_1|^2-|z_2|^2=0\}.
\end{equation}
The mechanism of solving equations is not well established in  noncommutative and nondivision algebra. Due to this reason,  it is difficult to consider the problem of matrix eigenvalue in such an  algebra.  The research of  eigenvalues over split quaternions is at incipient stage.

 According to (\ref{crex}), $A\in \bh_s^{n\times n}$ can be represented by  $A=A_1+A_2\bj$ with $A_1,A_2\in  \bc^{n\times n}$ and the complex  adjoint matrix of $A$ is defined by
 \begin{equation}\label{creA}\chi(A)=\left(\begin{array}{cc}
 		A_1&A_2  \\
 		\overline{A_2}&\overline{A_1}   \\
 	\end{array}\right).\end{equation}

  There are several attempts to understand the eigenvalues of a matrix $A$ over $\bh_s$ by using of its complex  adjoint matrix $\chi(A)$.  Alagoz  Oral and Yuce \cite{ala12} proposed Definition \ref{def1} in  split quaternions and showed that 
$A$ is invertible is equivalent to $A$ has  no zero eigenvalue. Erdogdu and Ozdemir \cite{ozd2013} showed (in Definition \ref{def1}) that $\sigma_r(A)\cap \bc=\sigma(\chi(A))$ and reduced $Ax=\lambda x$ to 
\begin{equation}\left(\begin{array}{cc}
	A_1-\lambda_1 E_n &A_2-\lambda_2 E_n  \\
	\overline{A_2-\lambda_2 E_n}&\overline{A_1-\lambda_1 E_n} \\
\end{array}\right)\left(\begin{array}{c}
	z_1  \\
	\overline{z_2}   \\
\end{array}\right)=\left(\begin{array}{c}
	0  \\
	0   \\
\end{array}\right),
\end{equation}
where $x=z_1+z_2\bj, \lambda= \lambda_1+ \lambda_2\bj$ with $z_i\in \bc^n,\lambda_i\in \bc,i=1,2$ and  $E_n$ is the
identity matrix of order $n$.

  In a division algebra,  the condition $x\neq 0$ is equivalent to saying that the vector $x$ contains at least one invertible component. Janovska and Opfer   have noticed this and refined Definition 1 as follows.

\begin{defi}(cf.\cite[definition 1.1]{Opfer16})\label{def2} Let $A\in\bh_s^{n\times n}$, $\lambda\in \bh_s$.   If  $\lambda$ holds the equation $$Ax =
	\lambda	x \mbox{ (resp. } Ax =
	x \lambda \mbox{)}$$ for   $x\notin Z(\bh_s)^n$, then $\lambda$  is called a left (resp. right) eigenvalue of $A$. The set of distinct left (resp. right) eigenvalues is called the left (resp. right)	spectrum of A, denoted $\sigma_l(A)$ (resp. $\sigma_r(A)$).
	Let 
	 \begin{equation}V_l(\lambda)=\{x\in \bh_s^{n}:Ax =
	\lambda x\}, \mbox{ if } \lambda\in \sigma_l(A)\end{equation}
	and 
 \begin{equation} V_r(\lambda)=\{x\in \bh_s^{n}:AX =
	x\lambda\},  \mbox{ if } \lambda\in \sigma_r(A).\end{equation}
\end{defi}

There  are perhaps two reasons for the requirement $x\notin Z(\bh_s)^n$ in the above definition. The first one is  Theorem 1.4 of \cite{Opfer16}, which says that if $A$ is a matrix over an arbitrary algebra with two different  eigenvalues $\lambda_1$ and $\lambda_2$ with respect to the same eigenvector $x$, then $\lambda_1=\lambda_2$.   The another reason is that the identity map $E_n:x\to x$ should have only one left eigenvalue $1$. That is, $E_nx=\lambda x$, which is equivalent to $(\lambda-1)x=0$, should imply that $\lambda=1$.

 Janovska and Opfer focused on  right eigenvalue problem in \cite{Opfer16}. One of the main results is to show that $A=\left(\begin{array}{cc}
 	1&\bi  \\
 	\bj&\bk   \\
 \end{array}\right)$ has no right eigenvalue \cite[Theorem 4.6]{Opfer16}. This example shows that there are matrices over $\bh_s$ without right eigenvalue. However, the necessary and sufficient condition of the existence of right eigenvalue  is still unknown.

Recently,  Falcao, Opfer and Janovska etc.\cite{Irene,Opfer17,Opfer18} have developed Niven's algorithm to solve the unilateral polynomials over $\bh_s$.  Cao and Chang \cite{cao,caoaxiom,caoarxiv} have derived 
explicit formulas for solving  the linear equation $ax=b$ and quadratic equation $ax^{2}+bx+c=0$.  These methods of solving split quaternion eqautions   make it possible for us  to consider the eigenvalue problem of matrices over $\bh_s$.

In this paper,  we focus on  the left eigenvalues of a split quaternionic matrix in the sense of  Definition \ref{def2}. 
By the formulas of solving the equations $ax=b$ and $ax^2+bx+c=0$ in split quaternions, we make an attempt to  understand the left eigenvalues for a split quaternion matrix of order $2$.
We show that  all triangle matrices always have left eigenvalues (see Theorems 3.1 and  3.2). For a general matrix $A=\left(\begin{array}{cc}
	a&b  \\
	c&d   \\
\end{array}\right)$, $A$ has a left eigenvalue if and only if  either  $cx^2+(d-a)x-b=0$ or $bx^2+(a-d)x-c=0$ is solvable (see Theorems 3.3).

\section{Several lemmas}
The following lemma is crucial to us.
\begin{lem}\label{lemobs}\begin{itemize}
		\item[(1)]
		For $p,q\in \bh_s$ and $A\in\bh_s^{n\times n}$, we have $$\sigma_l(pE_n + qA) = \{p + q\lambda: \lambda\in \sigma_l(A)\}.$$
		
		\item[(2)] Let $p,q\in \bh_s-Z(\bh_s)$. Then	$\lambda \in \sigma_l(A)$ if and only if $p\lambda q \in \sigma_l(pAq).$
		\item[(3)]
		If $x,y\in V_l(\lambda)$,  then $$A(x\mu_1+y\mu_2)=\lambda(x\mu_1+y\mu_2),\forall \mu_1,\mu_2\in \bh_s.$$
		This implies that $V_l(\lambda)$ is a right vector space over $\bh_s$.
	\end{itemize}
	
\end{lem}
\begin{proof}
	If $\lambda\in \sigma_l(A)$, then there exists an $x\notin Z(\bh_s)^n$ such that $Ax=\lambda x$. Therefore
	$$(pE_n + qA)x=px+qAx=(p + q\lambda)x.$$
	This implies that $p + q\lambda\in \sigma_l(pE_n + qA).$
	Also we have
	$$pAq(q^{-1}x)= p\lambda q (q^{-1}x),\, q^{-1}x\notin Z(\bh_s)^n.$$
	This implies that $p\lambda q\in \sigma_l(pAq).$
	It is obvious that $A(x\mu_1+y\mu_2)=\lambda(x\mu_1+y\mu_2)$ holds for $x,y\in V_l(\lambda)$. This implies that $V_l(\lambda)$ is a right vector space over $\bh_s$.
\end{proof}
By the above lemma, we can attribute the problem of left spectrum  to  solving the linear equation $ax=b$ and quadratic equation $ax^2+bx+c=0$ over $\bh_s$ in Section \ref{sect3}. We can solve such equations by two lemmas in what follows.

   The  Moore-Penrose inverse $a^+$  of  $a=z_1+z_2\bj,z_1,z_2\in \bc$ is defined in \cite{cao}  to be
$$a^+=\left\{\begin{array}{ll}
	0, & \hbox{if }\,\, $a=0;$ \\
	\frac{\overline{a}}{I_a}, & \hbox{if\, $I_a\neq 0$;} \\
	\frac{\overline{z_1}+z_2\bj}{4|z_1|^2}, & \hbox{if\, $I_a=0, a\neq 0$.} \\
\end{array}
\right.$$
For $a=z_1+z_2\bj\in  Z(\bh_s)-\{0\}$, we have the following equations:
\begin{equation}aa^+a=a,\  a^+aa^+=a^+,\  aa^+=\frac{1}{2}\big(1+\frac{z_2}{\overline{z_1}}\bj\big),
	\  a^+a=\frac{1}{2}\big(1+\frac{z_2}{z_1}\bj\big).\end{equation}

\begin{lem}(cf.\cite[Theorem 3.1]{cao})\label{lemge}
	Let $a,d\in \bh_s$. Then the equation $ax=d$ is  solvable
	if and only if $$aa^+d=d,$$
	in which case all solutions are given by $$x=a^+d+(1-a^+a)y,\quad \forall y=y_0+y_1\bi+y_2\bj+y_3\bk\in \bh_s \mbox{ with } y_i\in \br.$$
	The set of  solutions of  $ax=d$ is denoted by  \begin{equation}\label{sab}S_L(a,d):=\{x\in \bh_s:ax=d\}.\end{equation}
\end{lem}

We define the quasisimilar class of $q\in \bh_s$ as the following set \begin{equation}[q]=\{p\in \bh_s: \Re(p)=\Re(q), I_p=I_q\}.\end{equation}
For each class $[q]$, we define real coefficient quadratic polynomial
\begin{equation}\Psi_{[q]}(x)=x^2-2\Re(q)x+I_q.\end{equation}
According to \cite{caoarxiv} and references therein,  we have the following lemma.
\begin{lem}\label{eqQsolve}(c.f.\cite[Theorem 5.1]{caoarxiv})
	Let $ax^{2}+bx+c=0$ be a quadratic equation in $\bh_s$ with $a\neq 0$. Then it can be solved by the methods in \cite{caoaxiom,caoarxiv}.
	The set of  solutions of  $ax^{2}+bx+c=0$ is denoted by \begin{equation}\label{sabc}S_Q(a,b,c):=\{x\in \bh_s:ax^2+bx+c=0\}.
	\end{equation}
	If $a$ is invertible then \begin{equation}S_Q(a,b,c)=S_Q(1,a^{-1}b,a^{-1}c).\end{equation}	
	Furthermore,  suppose that the corresponding  companion  polynomial
	\begin{equation}\label{czpoly}c(x)=I_ax^4+2P_{ab}x^3+(2P_{ac}+I_b)x^2+2P_{bc}x+I_c\not\equiv 0.\end{equation}
	Let $\Psi_{[q]}(x)=x^2-Tx+N$ be a divisor of $c(x)$.  Then we have
	$$S_Q(a,b,c)=\bigcup_{[q]} \{S_L(Ta+b,aN-c)\cap [q]\}.$$
	
\end{lem}

\section{Left eigenvalue and its eigenvectors}\label{sect3}
To consider the left spectrum of $A=\left(\begin{array}{cc}
	a&b  \\
	c&d   \\
\end{array}\right)\in \bh_s^{2\times 2}$, we
will deal with the following cases:\\
\begin{equation}  b=0\mbox{ or }c=0;\quad   b\neq 0 \mbox{ and }c\neq 0.\end{equation}

\begin{thm}\label{thmb0}
	Let $A=\left(\begin{array}{cc}
		a&0  \\
		c&d   \\
	\end{array}\right)\in \bh_s^{2\times 2}$. Then we have the followings.
	\begin{itemize}
		\item[(1)]  If $S_L(a-d,c)\neq \emptyset$ then 
			\begin{equation} a\in \sigma_l(A),\,v=(1,x)^T\in V_l(a), \forall  x\in S_L(a-d,c).\end{equation} 
		
		\item[(2)] We have $0\in S_Q(c,d-a,0)\neq \emptyset $ and  \begin{equation}\lambda=cx+d\in  \sigma_l(A),\, v=(x,1)^T\in V_l(\lambda),\forall x\in S_Q(c,d-a,0).\end{equation} 
		 \item[(3)]  In particular, we have 
		 \begin{equation}d\in  \sigma_l(A),\,v=(x,1)^T\in V_l(d), \forall x\in S_L(c,0)\cap S_L(d-a,0);\end{equation} 
		  $a\in  \sigma_l(A)$ if and only if \begin{equation}S_L(a-d,c)\cup S_L(c,a-d)\neq \emptyset.\end{equation}
		   In this case		 
		    \begin{equation}\label{aeig1} a\in  \sigma_l(A)\mbox{ with eigenvectors }\left\{
		 	\begin{aligned}
		 		v=(x,1)^T,\, \forall x\in S_L(c,a-d)\mbox{ provided }S_L(c,a-d)\neq \emptyset;\\
		 		v=(1,x)^T,\, \forall x\in S_L(a-d,c)\mbox{ provided }S_L(a-d,c)\neq \emptyset.
		 	\end{aligned}
		 	\right.\end{equation}
		  
	\end{itemize}
\end{thm}

\begin{proof}
 	Suppose that  $\lambda\in \sigma_l(A)$ with an eigenvector $v\in V_l(\lambda)$.  By Definition \ref{def2},  $v$ contains at least one invertible component.  By Lemma \ref{lemobs}, $v\mu$  is a eigenvector for any  $\mu\in \bh_s-Z(\bh_s)$.  Choosing  a suitable $\mu\in \bh_s$, we may have two types of eigenvectors in  $V_l(\lambda)$: \begin{equation}v=(1,x)^T\mbox{  or }v=(x,1)^T.\end{equation} 
	
(1)\quad  Suppose that $A$ has the first type of  eigenvector in  $V_l(\lambda)$. Then we have 
	$$\left(\begin{array}{cc}
		a&0  \\
		c&d \\
	\end{array}\right)\left(\begin{array}{c}
		1  \\
		x   \\
	\end{array}\right)=\lambda\left(\begin{array}{c}
		1  \\
		x   \\
	\end{array}\right).$$
	That is \begin{equation}\label{rsym1b}\left\{
		\begin{aligned}
			a=\lambda,\\
			c+dx=\lambda x.
		\end{aligned}
		\right.\end{equation}
	Therefore\begin{equation}\label{b0eq}(a-d)x=c.\end{equation}
	If $S_L(a-d,c)\neq \emptyset$ then 	
	$ a\in \sigma_l(A),v=(1,x)^T\in V_l(a)$ for $\forall x\in S_L(a-d,c).$

\vspace{2mm}

(2)\quad  Suppose that $A$ has the second type of eigenvector  in  $V_l(\lambda)$. Then we have
	$$\left(\begin{array}{cc}
		a&0  \\
		c&d \\
	\end{array}\right)\left(\begin{array}{c}
		x  \\
		1   \\
	\end{array}\right)=\lambda\left(\begin{array}{c}
		x  \\
		1   \\
	\end{array}\right).$$
	That is \begin{equation}\label{rsym2}\left\{
		\begin{aligned}
			ax=\lambda x,\\
			cx+d=\lambda.
		\end{aligned}
		\right.\end{equation}
	Therefore\begin{equation}\label{b0eq1}cx^2+(d-a)x=0.\end{equation}
		It is obvious that $0\in S_Q(c,d-a,0)\neq \emptyset$. 
	Therefore $$\lambda=cx+d\in \sigma_l(A),\, v=(x,1)^T\in V_l(\lambda), \forall x\in S_Q(c,d-a,0).$$

	\vspace{2mm}

		(3)\quad 	
		By part (2), if $cx=0$ then  $d\in  \sigma_l(A)$ with $v=(x,1)^T\in V_l(d)$ for all $$x\in S_L(c,0)\cap S_Q(c,d-a,0)=S_L(c,0)\cap S_L(d-a,0).$$

		  If $cx+d=a$ is solvable, i.e.,   $S_L(c,a-d)\neq \emptyset $, then $a\in  \sigma_l(A)$ with $$v=(x,1)^T\in V_l(a),\,\forall x\in S_L(c,a-d)\cap S_Q(c,d-a,0)=S_L(c,a-d).$$
		  		   Combining with part (1), we have 
		  $a\in  \sigma_l(A)$ if and only if $S_L(a-d,c)\cup S_L(c,a-d)\neq \emptyset.$
		  Furthermore
		  		  \begin{equation*} a\in  \sigma_l(A)\mbox{ with eigenvectors }\left\{
		  	\begin{aligned}
		  		v=(x,1)^T,\, \forall x\in S_L(c,a-d);\\
		  	v=(1,x)^T,\, \forall x\in S_L(a-d,c).
		  	\end{aligned}
		  	\right.\end{equation*}

\end{proof}
\begin{exam}\label{exam3.1}
	Let	$A=\left(\begin{array}{cc}
		1+\bi&0  \\
		1+\bk&\bj+\bk\\
	\end{array}\right).$ That is, $a=1+\bi,c=1+\bk,d=\bj+\bk$.
	Note that $$0=(a-d)(a-d)^+c\neq c,\,\, 0=cc^+(a-d)\neq a-d.$$ By Lemma \ref{lemge}, we have  $$S_L(a-d,c)\cup S_L(c,a-d)= \emptyset.$$
	By Theorem \ref{thmb0}, we deduce that $a=1+\bi\notin  \sigma_l(A)$.
	
	\quad Let $p(x)=(1+\bk)x^2+(-1-\bi+\bj+\bk)x=0$. Then the   companion  polynomial is  $c(x)=-4x^3$. We have one pair $(T, N)=(0,0)$.
By  Lemma \ref{lemge}, we have
$$S_L(-1-\bi+\bj+\bk,0)=\{x\in\bh_s:x_0+x_1\bi+x_0\bj-x_1\bk,\forall x_0, x_1\in \br\}.
$$
and $$S_Q(c,d-a,0)=S_Q(1+\bk,-1-\bi+\bj+\bk,0)=\{x_1\bi-x_1\bk,\forall x_1\in \br\}.$$
	Hence
	$$\lambda=(1+\bk)(x_1\bi-x_1\bk)+\bj+\bk=(-1+\bi+\bj-\bk)x_1+\bj+\bk\in \sigma_l(A),v=(x_1\bi-x_1\bk,1)^T\in V_l(\lambda).$$

	Since $S_L(c,0)\cap S_L(d-a,0)=0,$
	we have  $d=\bj+\bk\in \sigma_l(A)$ with $v=(0,1)^T\in V_l(\bj+\bk)$.
\end{exam}

\begin{thm}\label{thmc0}
	Let $A=\left(\begin{array}{cc}
		a&b  \\
		0&d   \\
	\end{array}\right)\in \bh_s^{2\times 2}$. Then we have the followings.
	\begin{itemize}
		\item[(1)]  If $S_L(d-a,b)\neq \emptyset$ then 	
		\begin{equation} d\in \sigma_l(A),\,v=(x,1)^T\in V_l(d), \forall x\in S_L(d-a,b).\end{equation}

	\item[(2)] We have $0\in S_Q(b,a-d,0)\neq \emptyset$ and  \begin{equation}\lambda=a+bx\in  \sigma_l(A),\, v=(1,x)^T\in V_l(\lambda),\forall x\in S_Q(b,a-d,0).\end{equation}

	\item[(3)]  In particular, we have
	\begin{equation} a\in  \sigma_l(A),\,v=(1,x)^T\in V_l(a), \forall  x\in S_L(b,0)\cap S_L(a-d,0);\end{equation}   $d\in  \sigma_l(A)$ if and only if	\begin{equation}S_L(d-a,b)\cup S_L(b,d-a)\neq \emptyset.\end{equation} 
	In this case
	 \begin{equation} d\in  \sigma_l(A)\mbox{ with eigenvectors }\left\{
		\begin{aligned}
			v=(x,1)^T,\, \forall x\in S_L(d-a,b)\mbox{ provided }S_L(d-a,b)\neq \emptyset;\\
			v=(1,x)^T,\, \forall x\in S_L(b,d-a)\mbox{ provided }S_L(b,d-a)\neq \emptyset.
		\end{aligned}
		\right.\end{equation}
		\end{itemize}
\end{thm}

\begin{proof}
	(1)\quad Suppose that $\lambda\in \sigma_l(A)$ with an eigenvector $v=(x,1)^T$. That is 
	$$\left(\begin{array}{cc}
		a&b  \\
		0&d \\
	\end{array}\right)\left(\begin{array}{c}
		x  \\
		1   \\
	\end{array}\right)=\lambda\left(\begin{array}{c}
		x  \\
		1   \\
	\end{array}\right).$$
	Then \begin{equation}\label{rsym1c}\left\{
		\begin{aligned}
			ax+b=\lambda x,\\
			d=\lambda. 
		\end{aligned}
		\right.\end{equation}
	
	Hence \begin{equation}\label{b0eq4}(d-a)x=b.\end{equation}
 If $S_L(d-a,b)\neq \emptyset$ then 	
$ d\in \sigma_l(A),\,v=(x,1)^T\in V_l(d), \forall x\in S_L(d-a,b).$
	
(2)\quad Suppose that $\lambda\in \sigma_l(A)$ with an eigenvector $v=(1,x)^T$. Then
$$\left(\begin{array}{cc}
	a&b  \\
	0&d \\
\end{array}\right)\left(\begin{array}{c}
	1  \\
	x   \\
\end{array}\right)=\lambda\left(\begin{array}{c}
	1  \\
	x   \\
\end{array}\right).$$
That is \begin{equation}\label{rsymd}\left\{
	\begin{aligned}
		a+bx=\lambda,\\
		dx=\lambda x.
	\end{aligned}
	\right.\end{equation}

		Therefore\begin{equation}\label{b0eq3}bx^2+(a-d)x=0.\end{equation}
		Obviously, $0\in S_Q(b,a-d,0)$ and  \begin{equation}\lambda=a+bx\in  \sigma_l(A),\, v=(1,x)^T\in V_l(\lambda),\forall x\in S_Q(b,a-d,0).\end{equation} 
	 
	(3)\quad  By part (2), if $bx=0$ then  $a\in  \sigma_l(A)$ with $v=(1,x)^T\in V_l(a)$ for all $$x\in S_L(b,0)\cap S_Q(b,a-d,0)=S_L(b,0)\cap S_L(a-d,0).$$

	 If $a+bx=d$ is solvable, i.e.,   $S_L(b,d-a)\neq \emptyset $, then $d\in  \sigma_l(A)$ with $$v=(1,x)^T\in V_l(d),\,\forall x\in S_L(b,d-a)\cap S_Q(b,a-d,0)=S_L(b,d-a).$$
	 Combining with part (1), we have 
	 $d\in  \sigma_l(A)$ if and only if $S_L(d-a,b)\cup S_L(b,d-a)\neq \emptyset.$
	 Furthermore
	 \begin{equation*} d\in  \sigma_l(A)\mbox{ with eigenvectors }\left\{
	 	\begin{aligned}
	 		v=(x,1)^T,\, \forall x\in S_L(d-a,b);\\
	 		v=(1,x)^T,\, \forall x\in S_L(b,d-a).
	 	\end{aligned}
	 	\right.\end{equation*}

	\end{proof}

\begin{exam}\label{exam3.2}
	Let	$A=\left(\begin{array}{cc}
		-1+\bi&1+\bk  \\
		0&\bi+\bk\\
	\end{array}\right).$ That is, $a=-1+\bi, b=1+\bk, d=\bi+\bk$.
Since $$S_L(d-a,b)=S_L(b,d-a)=S_L(1+\bk,1+\bk)=\{x\in \bh_s:x=\frac{1+\bk}{2}+\frac{(1-\bk)y}{2},\forall y\in \bh_s\},$$
by Theorem \ref{thmc0}, we have  $$d=\bi+\bk\in  \sigma_l(A),\,v_1=(x,1)^T,\,v_2=(1,x)^T\in V_l(\bi+\bk),\forall x\in S_L(1+\bk,1+\bk).$$

	Note that $$S_Q(b,a-d,0)=\{x\in\bh_s:(1+k)(x^2-x)=0\}.$$ 
	By Lemma \ref{lemge}, we have
	$$S_Q(b,a-d,0)=\{x\in\bh_s:(x-\frac{1}{2})^2=\frac{1}{4}+(1-\bk)y,\forall y\in \bh_s\}.$$
	
	Let $x-\frac{1}{2}=x'=x_0+x_1\bi+x_2\bj+x_3\bk$ and $y=y_0+y_1\bi+y_2\bj+y_3\bk$. The above equation implies that
	\begin{equation}\left\{
		\begin{aligned}
			x_0^2-x_1^2+x_2^2+x_3^2&=&\frac{1}{4}+y_0-y_3,\\
			2x_0x_1&=&y_1-y_2,\\
			2x_0x_2&=&y_2-y_1,\\
			2x_0x_3&=&y_3-y_0.
		\end{aligned}
		\right.\end{equation}
	If $x_0=0$ then we need $y_1=y_2,y_0=y_3$ and $-x_1^2+x_2^2+x_3^2-\frac{1}{4}=0$;	
	if $x_0\neq 0$ then  $x_2=-x_1$ and $(x_0+x_3)^2=\frac{1}{4}.$  
	
		Hence $$S_Q(b,a-d,0)=S_1\cup S_2,$$ 
	where
	$$S_{1}=\{x=\frac{1}{2}+x_1\bi+x_2\bj+x_3\bk,\mbox{ where } -x_1^2+x_2^2+x_3^2-\frac{1}{4}=0\}$$
and
$$S_{2}=\{x=\frac{1}{2}+x_0+x_1\bi-x_1\bj+(\pm \frac{1}{2}-x_0)\bk,\forall x_0\neq 0,x_1\in \br\}.$$
	 Therefore 
	$$\lambda=-1+\bi+(1+\bk)x\in \sigma_l(A),\,v=(x,1)^T\in V_l(\lambda),\forall x\in S_Q(b,a-d,0).$$ 
\end{exam}

\begin{thm}\label{mainthm}
	Let $A=\left(\begin{array}{cc}
		a&b  \\
		c&d   \\
	\end{array}\right)\in \bh_s^{2\times 2}$ with $b\neq 0$ and $c\neq 0$.  Then
	\begin{itemize}
		\item[(1)]  If $S_Q(b,a-d,-c)\neq \emptyset$ then
		\begin{equation}\lambda=a+bx\in  \sigma_l(A),\,v=(1,x)^T\in V_l(\lambda),\forall x\in S_Q(b,a-d,-c).\end{equation}

		\item[(2)] If $S_Q(c,d-a,-b)\neq \emptyset$ then \begin{equation}d+cx\in  \sigma_l(A),\,v=(x,1)^T\in V_l(\lambda),\forall x\in S_Q(c,d-a,-b).\end{equation}
	\end{itemize}
\end{thm}

\begin{proof}
	Note that 
\begin{equation}\label{decompAs5}A=\left(\begin{array}{cc}
		a&b  \\
		c&d   \\
	\end{array}\right)=aE_2+\left(\begin{array}{cc}
		0&b  \\
		c&d-a \\
	\end{array}\right)=:aE_2+B.\end{equation} 
	Suppose that  $\mu\in \sigma_l(B)$ with $v\in V_l(\mu)$. Hence we may have two types of eigenvectors in $V_l(\mu)$: $$v=(1,x)^T\mbox{  or }v=(x,1)^T.$$

(1)\quad  Suppose that $\mu\in \sigma_l(B)$ with an eigenvector $v=(1,x)^T$.
Then
\begin{equation}\label{firstv}\left(\begin{array}{cc}
		0&b  \\
		c&d-a \\
	\end{array}\right)\left(\begin{array}{c}
		1  \\
		x   \\
	\end{array}\right)=\mu\left(\begin{array}{c}
		1  \\
		x   \\
	\end{array}\right).\end{equation} 
That is \begin{equation}\label{rsym51}\left\{
	\begin{aligned}
		bx=\mu,\\
		c+(d-a)x=\mu x.
	\end{aligned}
	\right.\end{equation}
Therefore\begin{equation}\label{beq}bx^2+(a-d)x-c=0.\end{equation}
By Lemma \ref{lemobs} and its proof,  if $S_Q(b,a-d,-c)\neq \emptyset$ then $$\lambda=a+\mu=a+bx\in  \sigma_l(A),\,v=(1,x)^T\in V_l(\lambda),\forall x\in S_Q(b,a-d,-c).$$

(2)\quad Suppose that $\mu\in \sigma_l(B)$ with an eigenvector $v=(x,1)^T$.
Then
$$\left(\begin{array}{cc}
	0&b  \\
	c&d-a \\
\end{array}\right)\left(\begin{array}{c}
	x  \\
	1   \\
\end{array}\right)=\mu\left(\begin{array}{c}
	x  \\
	1   \\
\end{array}\right).$$
That is \begin{equation}\label{rsyme1}\left\{
	\begin{aligned}
		b=\mu x,\\
		cx+(d-a)=\mu.
	\end{aligned}
	\right.\end{equation}
Therefore\begin{equation}\label{ceq5}cx^2+(d-a)x-b=0.\end{equation}
By Lemma \ref{lemobs},  if $S_Q(c,d-a,-b)\neq \emptyset$ then
 $$\lambda=a+\mu=a+[cx+(d-a)]=d+cx\in  \sigma_l(A),\,v=(x,1)^T, \forall x\in S_Q(c,d-a,-b).$$
\end{proof}

\begin{exam}\label{exam3.3}
	Let	$A=\left(\begin{array}{cc}
		1&\bi  \\
		\bj&\bk\\
	\end{array}\right).$ That is, $a=1, b=\bi,c=\bj, d=\bk$.\\
		Let $p(x)=x^2+(-\bi-\bj)x+\bk=0$. Then the   companion  polynomial is  $c(x)=x^4-1=(x^2+1)(x^2-1)$. We have two pairs $(T, N)=(0,1)$ and $(T, N)=(0,-1)$.
		  By  Lemma \ref{lemge}, we have
		$$S_L(-\bi-\bj,1-\bk)=\emptyset
		$$
		and $$S_L(-\bi-\bj,-1-\bk)=\{\frac{1}{2}(y_0+y_3)+[-\frac{1}{2}+\frac{1}{2}(y_1+y_2)]\bi+[\frac{1}{2}+\frac{1}{2}(y_1+y_2)]\bj+\frac{1}{2}(y_0+y_3)\bk\}.$$
		
		Hence $$S_L(-\bi-\bj,-1-\bk) \cap \{x\in \bh_s: \Re(x)=0, I_x=-1\}=\{\bj\}.$$
		By Lemma \ref{eqQsolve},	
		$$S_Q(1,-\bi-\bj,\bk)=\{\bj\}.$$

		Note that $$S_Q(b,a-d,-c)=\{x\in\bh_s:\bi x^2+(1-\bk)x-\bj=0\}=\{x\in\bh_s: x^2+(-\bi-\bj)x+\bk=0\}=\{\bj\}$$ 
				and 
		$$S_Q(c,d-a,-b)=\{x\in\bh_s:\bj x^2+(\bk-1)x-\bi=0\}=\{x\in\bh_s: x^2+(-\bi-\bj)x+\bk=0\}=\{\bj\}.$$
	By Theorem \ref{mainthm}, $A$ has left eigenvalue $\lambda=1+\bk$ with eigenvectors $v=(1,\bj)^T$ and $v=(\bj,1)^T$. We mention that  $A$  has no right eigenvalue (see \cite[Theorem 4.6]{Opfer16}). 
\end{exam}

\begin{exam}\label{exam3.4}
	Let	$A=\left(\begin{array}{cc}
		1& 1\\
		-3-\bi-\bj-\bk&1\\
	\end{array}\right).$ That is, $a=b=d=1, c=-3-\bi-\bj-\bk$.\\

	Example 2 in \cite{caoaxiom} implies that $$S_Q(b,a-d,-c)=\{x\in\bh_s: x^2+3+\bi+\bj+\bk=0\}=\emptyset$$ 
	and 
	$$S_Q(c,d-a,-b)=\{x\in\bh_s:(-3-\bi-\bj-\bk) x^2-1=0\}=\{x\in\bh_s: x^2=\frac{-3+\bi+\bj+\bk}{8}\}.$$
	Let $w=\frac{-3+\bi+\bj+\bk}{8}$. Since $I_w=\frac{1}{8}$ and $w_0+\sqrt{I_w}<0$,by	Lemma 1 (2) in \cite{caoaxiom}, we have 
	$$S_Q(c,d-a,-b)=\emptyset.$$
		By Theorem \ref{mainthm}, $A$ has no left eigenvalue. 
\end{exam}

\vspace{2mm}
{\bf Acknowledgments.}\quad
This work is supported by National Natural Science Foundation of China (11871379),
the Innovation Project of Department of Education of Guangdong Province (2018KTSCX231) and
Key project of  Natural Science Foundation  of Guangdong Province Universities (2019KZDXM025).


\begin{thebibliography}{99}
\bibitem{ala12}	Y. Alagoz, K. Oral, S. Yuce, Split Quaternion Matrices. Miskolc Math. Notes 13(2012) 223-232.

	
\bibitem{cao} W. Cao,  Z. Chang,	  Moore-Penrose inverse of split quaternion, Linear and Multilinear Algebra, doi.org/10.1080/03081087.2020.1769015.	

\bibitem{caoaxiom}W. Cao, Quadratic equation in split quaternions. Axioms. 11(5)(2022) 188. 

\bibitem{caoarxiv} W. Cao,	Quadratic unilateral polynomials over split quaternions,	arXiv:2205.03605. 

\bibitem{cohn77}P. Cohn, Skew field constructions, London Mathematical Society Note Series 27, Cambridge
University, Cambridge, 1977.


%

\bibitem{eilen44} S. Eilenberg, I. Niven, The fundamental theoreom of algebra for quaternions, Bull. Amer. Math. Soc.
50 (1944) 246-248.


\bibitem{ozd2013} M. Erdogdu, M. Ozdemir, On eigenvalues of split quaternion matrices, Adv.
Appl. Clifford Algebras. 23(2013) 614-623.

\bibitem{Irene} M. Falcao, F. Miranda, R. Severino, M. Soares, The number of zeros of unilateral polynomials over coquaternions revisited, arXiv.1703.10986	

\bibitem{hs02} L. Huang, W. So, Quadratic formulas for quaternions, Appl. Math. Lett. 15(2002) 533-540.




%




%









\bibitem{hs2001}
 L. Huang, W. So,
On left eigenvalues of a quaternionic matrix,
 Linear Algebra Appl. 323(2001) 105-116.


\bibitem{Opfer16}D. Janovska, G. Opfer, Matrices over nondivision algebras
without eigenvalues, Adv. Appl. Clifford Algebras  26(2016) 591-612.

 
 \bibitem{Opfer17}G. Opfer, Niven’s algorithm applied to the roots of the companion polynomial
 over $R^4$ algebras. Adv. Appl. Clifford Algebras,  27(2017) 2659-2675.
 
 
 \bibitem{Opfer18} D. Janovska, G. Opfer, The relation between the companion
 matrix and the companion polynomial in $R^4$ algebras, Adv. Appl. Clifford Algebras, 28(2018)76.

\bibitem{lee49} H. Lee, Eigenvalues and canonical forms of matrices with quaternion coefficients, Proc. Roy. Irish
Acad. 52A (1949) 253-260.


\bibitem{so05}W. So,  Quaternionic left eigenvalue problem. Southeast Asian Bull.
Math. 29(3)(2005) 555-565.

\bibitem{wood85} R. Wood, Quaternionic eigenvalues, Bull. London Math. Soc. 17(1985) 137-138.
\end{thebibliography}
\end{document}